\documentclass{amsart}
\usepackage{amssymb}

\newtheorem{theorem}{Theorem}[section]
\newtheorem{lemma}[theorem]{Lemma}
\newtheorem{corollary}[theorem]{Corollary}
\newtheorem{conjecture}[theorem]{Conjecture}

\newtheorem{example}[theorem]{Example}

\newtheorem{definition}[theorem]{Definition}
\renewcommand{\wr}{\mathop{\mathrm{wr}}}

\def\F#1{{\bf F}(#1)}
\def\O#1#2{{{\bf O}}_{{#1}}({{#2}})}
\def\FF#1#2{{\bf F}_{{#1}}(#2)}
\def\cent#1#2{{\bf C}_{#1}(#2)}

\def\R#1{{\bf R}(#1)}
\def\RR#1#2{{\bf R}_{{#1}}(#2)}
\title[Fitting height of factorisable groups]{On the Fitting height of soluble groups admitting a coprime factorisation}

\author[C. Casolo]{Carlo Casolo}
\address{Carlo Casolo, Dipartimento di Matematica ``Ulisse
  Dini'', \newline University of Firenze, Viale Morgagni 67/a, 50134 Firenze Italy}
\email{casolo@math.unifi.it}

\author[E. Jabara]{Enrico Jabara}
\address{Enrico Jabara, Dipartimento di Filosofia e Beni Culturali, \newline
University  C\'a Foscari,  Dorsoduro 3484/D – 30123 Venezia, I-30100 Venezia, Italy. }
\email{jabara@unive.it}
\author[P. Spiga]{Pablo Spiga}
\address{Pablo Spiga, Dipartimento di Matematica Pura e Applicata,\newline
 University of Milano-Bicocca, Via Cozzi 53, 20126 Milano, Italy.}
\email{pablo.spiga@unimib.it}

\thanks{Address correspondence to Pablo Spiga: pablo.spiga@unimib.it.}



\begin{document}

\begin{abstract}
In this paper we are concerned with finite soluble groups $G$ admitting a factorisation $G=AB$, with $A$ and $B$ proper subgroups having coprime order. We are interested in bounding the Fitting height of $G$ in terms of some group-invariants of $A$ and $B$: including the Fitting heights and the derived lengths.
\end{abstract}

\maketitle

\section{Introduction}\label{intro}

In this paper, all groups considered are finite and soluble, and hence the word ``group'' should always be understood as ``finite soluble group''.

We investigate groups $G$ in which a \textit{factorisation} $$G=AB=\{ab\mid a\in A,\,b\in B\}$$
 with $A$ and $B$ subgroups of $G$ of coprime order is given. We are interested in obtaining some upper bounds on the \textit{Fitting height} $h(G)$ of $G$, in terms of the Fitting heights ($h(A)$ and $h(B)$) and of the \textit{derived lengths} ($d(A)$ and $d(B)$) of $A$ and $B$. (Our notation is standard,  see Section~\ref{1.1} for undefined terminology.)

\begin{theorem}\label{thrmA}
Let $G=AB$ be a finite soluble group factorised by its proper subgroups $A$ and $B$ with $\gcd(|A|,|B|)=1$. If $|B|$ is odd, then 
\begin{equation}\label{eq:1}
h(G)\leq h(A)+h(B)+2d(B)-1.
\end{equation} 
If $B$ is nilpotent, then 
\begin{equation}\label{eq:2}
h(G)\leq h(A)+2d(B).
\end{equation}
\end{theorem}

Before  continuing with our discussion we need to introduce some notation. Given a group $G$, we write
\begin{equation*}
\delta(G):=\max\{d(S)\mid S \textrm{ Sylow subgroup of }G\},
\end{equation*}
that is, $\delta(G)$ is the maximal derived length of the Sylow subgroups of $G$. We also bound the Fitting height of $G$ in terms of the group-invariants $\delta(A)$ and $\delta(B)$.

\begin{theorem}\label{thrmB}
Let $G=AB$ be a finite soluble group factorised by its proper subgroups $A$ and $B$ with $\gcd(|A|,|B|)=1$. Then 
\begin{equation*}%\label{eq:4}
h(G)\leq h(A)+(2\delta(B)+1)h(B)-1.
\end{equation*}
\end{theorem}

Both Theorems~\ref{thrmA} and~\ref{thrmB} extend and generalise some well-known results on groups admitting a factorisation with subgroups of coprime order, see for example the two monographs~\cite[Chapter~$2$]{AFD} and~\cite[pages~$133$--$135$]{BB}. Observe that when $A$ and $B$ are both nilpotent, we have $h(A)=h(B)=1$ and the inequality in Theorem~\ref{thrmB} specialises to the inequality of the main result in~\cite{Gemma}.

When $B$ is nilpotent, we have $\delta(B)=d(B)$ and $h(B)=1$, and thus Theorem~\ref{thrmA}~\eqref{eq:2} follows immediately from Theorem~\ref{thrmB}.   

The hypothesis of $|B|$ being odd in Theorem~\ref{thrmA}~\eqref{eq:1} is important in our proof because at a critical juncture we apply a remarkable theorem of Kazarin~\cite{Ka} (which requires $B$ having odd order). However, we believe that our hypothesis is only factitious and in fact we pose the following:
\begin{conjecture}
Let $G=AB$ be a finite soluble group factorised by its proper subgroups $A$ and $B$ with $\gcd(|A|,|B|)=1$. Then $$
h(G)\leq h(A)+h(B)+2d(B)-1.$$
\end{conjecture}

We  also prove:
\begin{theorem}\label{thrmC}
Let $G=AB$ be a finite soluble group factorised by its proper subgroups $A$ and $B$ with $\gcd(|A|,|B|)=1$. Then 
\begin{equation*}%\label{eq:4}
h(G)\leq h(A)\delta(A)+h(B)\delta(B).
\end{equation*}
\end{theorem}

Finally, with an immediate application of Theorem~\ref{thrmA} and of the machinery developed in Section~\ref{section3}, we prove: 
\begin{corollary}\label{corcor}
Let $G=AB$ be a finite soluble group factorised by its proper subgroups $A$ and $B$ with $\gcd(|A|,|B|)=1$. For each $p\in \pi(B)$, let $B_p$ be a Sylow $p$-subgroup of $B$. Then $$h(G)\leq h(A)+2\sum_{p\in\pi(B)}d(B_p).$$ In particular, $h(G)\leq h(A)+2|\pi(B)|\delta(B)$.
\end{corollary}

In Section~\ref{1.1} we introduce some basic notation and some preliminary results that we use throughout the whole paper. In Section~\ref{section3} we present our main tool (the \textit{towers} as defined by Turull~\cite{Tu}) and we prove some auxiliary results. Section~\ref{section4} is dedicated to the proof of Theorems~\ref{thrmA} and~\ref{thrmB} and of Corollary~\ref{corcor}. The proof of Theorem~\ref{thrmC} (which requires a slightly different machinery) is postponed to Section~\ref{section5}.

\section{Notation and preliminary results}\label{1.1}

Given a group $G$, we denote by $\F G$ the \textit{Fitting subgroup} of $G$ (that is, the largest normal nilpotent subgroup of $G$). Moreover, the Fitting series of $G$ is defined inductively by $\FF 0 G:=1$ and $\FF {i+1}G/\FF i G:=\F{G/\FF i G}$, for every $i\ge 0$. Clearly, $\FF i G<\FF {i+1} G$ when $\FF i G<G$, and the minimum natural number $h$ with $\FF h G=G$ is called the \textit{Fitting height} (or Fitting length) of $G$ and is denoted by $h(G)$. Similarly, the \textit{derived length} of $G$ is indicated by $d(G)$.

We let $|G|$ denote the order of $G$ and we let $\pi(G)$ denote the set of prime divisors of $|G|$. 
Given a prime number $p$, we write $G_p$ for a Sylow $p$-subgroup of $G$. A \textit{Sylow basis} of $G$ is a family $\{G_p\}_{p\in \pi(G)}$ of Sylow subgroups of $G$ such that
$G_{p}G_q=G_qG_p$ for any $p,q\in \pi(G)$. By a pioneering result of Philip Hall~\cite[$9.1.7$,~$9.1.8$ and $9.2.1$~(ii)]{Rob}, every (finite soluble) group has a Sylow basis. In particular, for every set of primes $\pi$, $G$ contains a Hall $\pi$-subgroup, which will be denoted by $G_{\pi}$.

Given a set $\pi$ of prime numbers, we set $\pi':=\{p\textrm{ prime}\mid p\notin \pi\}$. Moreover, when $\pi=\{p\}$, for simplicity we write $p'$ for $\pi'$. As usual, $\O \pi G$ is the largest normal $\pi$-subgroup of $G$ and the upper $\pi'\pi$-\textit{series} of $G$ is generated by  applying ${\bf O}_{\pi'}$ and ${\bf O}_\pi$ (in this order) repeatedly to $G$, that is, the series
$1=P_0\leq N_0\leq P_1\leq N_1\leq \cdots \leq P_i\leq N_i\leq \cdots $ defined by
\[
N_i/P_i:=\O{\pi'}{G/P_i}\quad\textrm{and}\quad P_{i+1}/N_i:=\O \pi {G/N_i}.
\]
This is a series of characteristic subgroups having factor groups $\pi'$- and $\pi$-groups, alternately. The minimum natural number $\ell$ such that the $\pi'\pi$-series terminates is named the $\pi$-\textit{length} of $G$ and denoted by $\ell_\pi(G)$. When $\pi=\{p\}$, we write simply $\O p G$ and $\ell_p(G)$.

We first state a basic elementary result which will be used repeatedly and without comment.

\begin{lemma}\label{lemma:2.1}
Let $G=AB$ be a group factorised by $A$ and $B$ with $\gcd(|A|,|B|)=1$. Then there exists a Sylow basis $\{G_p\}_{p\in \pi(G)}$ with $A=\prod_{p\in \pi(A)}G_p$ and $B=\prod_{p\in \pi(B)}G_p$.
\end{lemma}
\begin{proof}
From~\cite[Lemma~$1.3.2$]{AFD}, we see that for every $p\in \pi(G)$ there exists a Hall $p'$-subgroup $A_{p'}$ of $A$ and a Hall $p'$-subgroup $B_{p'}$ of $B$ such that $A_{p'}B_{p'}$ is a Hall $p'$-subgroup of $G$. Now, for each $p\in \pi(G)$, define $G_p:=\bigcap_{q\in \pi(G)\setminus\{p\}}A_{q'}B_{q'}$. A computation shows that $\{G_p\}_{p\in \pi(G)}$ is a Sylow basis of $G$ (see for example~\cite[$9.2.1$]{Rob}). Moreover, $A=\prod_{p\in \pi(A)}G_p$ and $B=\prod_{p\in \pi(B)}G_p$.
\end{proof}

The next two results are crucial for our proofs of Theorems~\ref{thrmA} and~\ref{thrmB}. 

\begin{theorem}\label{thrm2.2}
Let $G$ be a group and let $p$ be a prime. Then $\ell_p(G)\leq d(G_p)$.
\end{theorem}
\begin{proof}
When $p$ is odd, this is~\cite[Theorem~$A$~(i)]{HH}. The analogous result for $p=2$ is proved in~\cite{Br}.
\end{proof}

Kazarin~\cite{Ka} has proved Theorem~\ref{thrm2.2} for arbitrary sets  of primes $\pi$ with $2\notin \pi$. We state this generalisation in a form tailored to our needs.

\begin{theorem}\label{thrm2.3}
Let $G$ be a group and let $\pi$ be a set of primes. If $2\notin \pi$ or if $G_\pi$ is nilpotent, then $\ell_\pi(G)\leq d(G_\pi)$.
\end{theorem}
\begin{proof}
When $2\notin \pi$, this is the main result of~\cite{Ka} (see also~\cite[Theorem~$1.7.20$]{BB}). When $G_\pi$ is nilpotent, the proof follows from Theorem~\ref{thrm2.2}. 
\end{proof}

\section{Our toolkit: towers}\label{section3}
We start this section with a pivotal definition introduced by Turull~\cite{Tu}. (The definition of $B$-\textit{tower} in~\cite[Definition~$1.1$]{Tu} is actually more general then the one we give here and coincides with ours when $B=1$.)

\begin{definition}\label{Ttower}{\rm 
Let $G$ be a group. A family $\mathfrak{T} := (P_i\mid i\in \{1,\ldots,h\})$ is said to be a \textit{tower of length} $h$ of $G$ if the following are satisfied.
\begin{enumerate}
\item $P_i$ is a $p_i$-subgroup of $G$ and $p_i\in \pi(G)$.
\item If $1\leq i\leq j\leq h$, then $P_i$ normalises $P_j$.
\item Define inductively $\overline{P_h}:=P_h$, and $\overline{P_i}:=P_i/\cent {P_i}{\overline{P_{i+1}}}$ for  $i\in \{1,\ldots,h-1\}$. Then $\overline{P_i}\neq 1$, for every $i\in \{1,\ldots,h\}$.
\item $p_i\neq p_{i+1}$, for every $i\in\{1,\ldots,h-1\}$.
\end{enumerate}
}
\end{definition}

 A concept that resembles the definition of tower was orinigally introduced by Dade in~\cite{Dade} for investigating the Fitting height of a group.  
The relationship between Fitting height and towers was uncovered by Turull.
\begin{lemma}[{{\cite[Lemma~$1.9$]{Tu}}}]\label{lemma31}Let $G$ be a group. Then
\[h(G)=\max\{h\mid G \textrm{ admits a tower of length }h\}.\]
\end{lemma}
In view of Lemma~\ref{lemma31} we give the following:
\begin{definition}\label{Fittingtower}{\rm We say that the tower $\mathfrak{T}$ of $G$ a \textit{Fitting tower} if $\mathfrak{T}$ has length $h(G)$.}
\end{definition}

%\begin{definition}\label{Ttower1}{\rm 
%Let $G$ be a group and let $\mathfrak{T}_1=(P_i)_{i=1}^{h_1}$ and $\mathfrak{T}_2=(Q_i)_{i=1}^{h_2}$ be two towers of $G$ of length $h_1$ and $h_2$, respectively. We say that $\mathfrak{T}_1$ is \textit{contained} in $\mathfrak{T}_2$ if there exists an increasing function $f:\{1,\ldots,h_1\}\to \{1,\ldots,h_2\}$ with $P_i\leq Q_{f(i)}$, for every $i\in \{1,\ldots,h_1\}$.
%}
%\end{definition}

%For the proof of Lemma~\ref{lemma33}, we need the following result from~\cite{Tu}.
%\begin{lemma}[{{\cite[Lemma~$1.5$]{Tu}}}]\label{lemma32}Let $G$ be a group, let $(P_i\mid i\in \{1,\ldots,h\})$ be a family of subgroups of $G$ satisfying the conditions~$(1)$,~$(2)$ and~$(3)$ in Definition~$\ref{Ttower}$, let $1\leq h_0\leq h$ and let $f:\{1,\ldots,h_0\}\to \{1,\ldots,h\}$ be an increasing function. Then $(P_{f(i)}\mid i\in \{1,\ldots,h_0\})$   satisfies the conditions~$(1)$,~$(2)$ and~$(3)$ in Definition~$\ref{Ttower}$.
%\end{lemma}

The following is an easy consequence of~\cite[Lemma~$1.5$]{Tu}. 
For simplifying the  notation, given a $p$-group $P$, we write $\pi^*(P)=p$ when $P\neq 1$, and $\pi^*(P)=1$ when $P=1$. Observe that when $P\neq 1$ we have $\pi(P)=\{\pi^*(P)\}$.

\begin{lemma}\label{lemma33}
Let $G$ be a group, let $\mathfrak{T}=(P_i\mid i\in \{1,\ldots,h\})$ be a tower of $G$,  let $j\in \{1,\ldots,h\}$, let $s\geq 0$ be an integer and let $\mathfrak{T}'=(P_i\mid i\in \{1,\ldots,h\}\setminus\{j,j+1,\ldots,j+s-1,j+s\})$. Then either $\mathfrak{T}'$ is a tower of $G$, or $1<j\leq 
j+s<h$ and $\pi^*(P_{j-1})=\pi^*(P_{j+s+1})$.
\end{lemma}
\begin{proof}
Lemma~$1.5$ in~\cite{Tu} says that, for every $h_0$ with $1\leq h_0\leq h$ and for every increasing function $f:\{1,\ldots,h_0\}\to \{1,\ldots,h\}$, the family
$(P_{f(i)}\mid i\in \{1,\ldots,h_0\})$ satisfies the conditions ~$(1)$,~$(2)$ and~$(3)$ in Definition~$\ref{Ttower}$. Applying this with  $h_0:=h-s-1$ and with $f:\{1,\ldots,h_0\}\to \{1,\ldots,h\}$  defined by 
\[
f(i)=\begin{cases}
i&\textrm{if }1\leq i< j,\\
i+s+1&\textrm{if }j\leq i\leq h_0,
\end{cases}
\]
we obtain  that $\mathfrak{T}'$ satisfies the conditions~$(1)$,~$(2)$ and~$(3)$ of Definition~\ref{Ttower}. As $\mathfrak{T}$ satisfies Definition~\ref{Ttower}~$(4)$, we immediately get that either $\mathfrak{T'}$  satisfies also~$(4)$ (and hence is a tower of $G$), or $1<j\leq j+s<h$ and $\pi^*(P_{j-1})=\pi^*(P_{j+s+1})$. 
\end{proof}

\begin{definition}{\rm 
Let $G$ be a group, let $\mathfrak{T}=(P_i\mid i\in \{1,\ldots,h\})$ be a tower of $G$ and let $\sigma$ be a set of primes. We set $$\nu_\sigma(\mathfrak{T}):=|\{i\in \{1,\ldots,h\}\mid \pi^*(P_i)\in \sigma\}|.$$ Clearly, $\nu_\sigma(\mathfrak{T})=0$ when $\sigma$ has no element in common with $\{\pi^*(P_1),\ldots,\pi^*(P_h)\}$.

Now, set $P_0:=1$ and $P_{h+1}:=1$. For $i,j\in \{1,\ldots,h\}$ with $i\leq j$, the sequence $(P_\ell\mid i\leq \ell\leq j)$ of consecutive elements of $\mathfrak{T}$ is said to be a $\sigma$-\textit{block} if  
\begin{itemize}
\item $\pi^*(P_{i+s})\in \sigma$ for every $s$ with $0\leq s\leq j-i$, and
\item $\pi^*(P_{i-1})\notin\sigma$, $\pi^*(P_{j+1})\notin \sigma$.
\end{itemize}
Moreover, we denote by $\beta_\sigma(\mathfrak{T})$ the number of $\sigma$-blocks of $\mathfrak{T}$.}
\end{definition}

The main result of this section is  Lemma~\ref{lemma34}: before proceeding to its proof we single out  two basic observations.

\begin{lemma}\label{elementary}
Let $\mathfrak{T}=(P_i\mid i\in \{1,\ldots,h\})$ be a tower of $G$. Then, for $j\in \{1,\ldots,h-1\}$, we have $\cent {P_j}{P_h}\le \cent {P_j}{\overline{P_{j+1}}}$.
\end{lemma}
\begin{proof}
We argue by induction on $h-j$. If $j=h-1$, then $\overline{P_h}=P_h$ and hence there is nothing to prove. Suppose $h-j>1$ and set $R:=\cent {P_j}{P_h}$. We have $[R,P_h,P_{j+1}]=1$, and also $[P_h,P_{j+1},R]\leq [P_h,R]=1$ by Definition~\ref{Ttower}~$(2)$. Thus the Three Subgroups Lemma yields $[P_{j+1},R,P_h]=1$, that is, $[P_{j+1},R]\leq \cent {P_{j+1}}{P_h}$. Now the inductive hypothesis gives $[P_{j+1},R]\leq \cent {P_{j+1}}{\overline{P_{j+2}}}$, and hence $[\overline{P_{j+1}},R]=1$. Therefore $\cent {P_j}{P_h}=R\leq \cent {P_j}{\overline{P_{j+1}}}$.
\end{proof}

\begin{lemma}\label{elementaryy}
Let $\mathfrak{T}=(P_i\mid i\in \{1,\ldots,h\})$ be a tower of $G$ and let $N$ be a normal subgroup of $G$ with 
\begin{equation}\label{eq1}
P_j\cap N\leq \cent {P_j}{P_h},
\end{equation}
for every $j\in \{1,\ldots,h-1\}$. Then $\mathfrak{T}':=(P_iN/N\mid i\in 1,\ldots,h-1\})$ is a tower of $G/N$.
\end{lemma}
\begin{proof}
From~\eqref{eq1} and Lemma~\ref{elementary}, we have $P_j\cap N\leq \cent {P_j}{\overline{P_{j+1}}}$ for $j<h$. Set $R_h:=1$, and set $R_j:=\cent {P_j}{\overline{P_{j+1}}}$ for $j<h$. Thus $\overline{P_j}=P_j/R_j$, for every $j$. 

Now, for $j<h$, we have
$$ P_j\cap N= R_j\cap N$$ and hence
\begin{equation}\label{Pjbar}
\frac{P_jN}{R_jN}=\frac{P_j(R_jN)}{R_jN}\cong \frac{P_j}{P_j\cap R_jN}=\frac{P_j}{R_j(P_j\cap N)}=\frac{P_j}{R_j(R_j\cap N)}=\frac{P_j}{R_j}=\overline{P_j}.
\end{equation}

For each $j\in \{1,\ldots,h-1\}$, set $Q_j:=P_jN/N$, and define $\overline{Q_{h-1}}:=Q_{h-1}$, and $\overline{Q_j}:=Q_j/\cent {Q_j}{\overline{Q_{j+1}}}$ for $j<h-1$. In particular, for each $j\in \{1,\ldots,h-2\}$, there exists $L_j\leq P_j$ with $\cent {Q_j}{\overline{Q_{j+1}}}=L_jN/N$. Moreover, set $L_{h-1}:=1$. 

We show (by induction on $h-j$) that $L_j\leq R_j$, for each $j\in \{1,\ldots,h-1\}$. If $h-j=1$, then $L_j=L_{h-1}=1\leq R_{h-1}=R_j$. Assume then that $h-j>1$ and let $x\in L_{j}$. As $[xN,\overline{Q_{j+1}}]=1$, we get $[xN,Q_{j+1}]\leq \cent {Q_{j+1}}{\overline{Q_{j+2}}}=L_{j+1}N/N$ when $h-j>2$, and $[xN,Q_{j+1}]=1$ when $h-j=2$.  In both cases, applying the inductive hypothesis, we obtain
$$[x,Q_{j+1}]\leq \frac{L_{j+1}N}{N}\leq \frac{R_{j+1}N}{N}.$$
This gives $$[x,P_{j+1}]\leq P_{j+1}\cap R_{j+1}N=R_{j+1}(P_{j+1}\cap N).$$ Combining~\eqref{eq1}, Lemma~\ref{elementary} and the definition of $R_{j+1}$, we have $P_{j+1}\cap N\leq \cent {P_{j+1}}{P_h}\leq \cent {P_{j+1}}{\overline{P_{j+2}}}=R_{j+1}$. Therefore $[x,P_{j+1}]\leq R_{j+1}$ and hence $x\in \cent {P_{j}}{P_{j+1}/R_{j+1}}=\cent {P_j}{\overline{P_{j+1}}}=R_j$. Thus $L_j\leq R_j$ and the induction is proved.

Observe that
\begin{equation}\label{Qjbar}
\overline{Q_j}=\frac{P_jN/N}{L_{j}N/N}\cong \frac{P_jN}{L_jN}.
\end{equation}
As $L_jN\le R_jN \le P_jN$, from~\eqref{Pjbar} and~\eqref{Qjbar}, we see that
$\overline{P_j}$ is an epimorphic image of $\overline Q_j$. Finally, since $\mathfrak{T}$ is a tower of $G$, it follows immediately that $\mathfrak{T}'$ is a tower of $G/N$. 
\end{proof}

Given a tower $\mathfrak{T}=(P_i\mid i\in \{1,\ldots,h\})$ and $j\in \{1,\ldots,h\}$, we set $T_j:=P_hP_{h-1}\cdots P_j$. Observe that from Definition~\ref{Ttower}~(2), we have $T_j\unlhd T_1$.
 
We are now ready to prove one of the main tools of our paper. 
\begin{lemma}\label{lemma34}
Let $G$ be a group, let $\sigma$ be a non-empty subset of $\pi(G)$, let $A$ be a Hall $\sigma$-subgroup of $G$ and let $\mathfrak{T}:=(P_i\mid i\in \{1,\ldots,h\})$ be a tower of $G$. Then
\begin{enumerate}
\item $h(A)\geq \nu_\sigma(\mathfrak{T})-\beta_\sigma(\mathfrak{T})+1$, and
\item $\ell_\sigma(G)\geq\beta_\sigma(\mathfrak{T})$.
\end{enumerate} 
\end{lemma}
\begin{proof}Observe that $h(A),\ell_\sigma(G)\geq 1$ because  $\emptyset\neq\sigma\subseteq \pi(G)$. In particular, we may assume that $\nu_\sigma(\mathfrak{T}),\beta_{\sigma}(\mathfrak{T})\neq 0$ and hence $\sigma_0:=\sigma\cap \{\pi^*(P_i)\mid 1\leq i\leq h\}\neq\emptyset$. Let $A_0$ be a Hall $\sigma_0$-subgroup of $T_1$. Observe that $\mathfrak{T}$ is a tower of $T_1$ and that the hypothesis of this lemma are satisfied with $(G,\sigma,A)$ replaced by $(T_1,\sigma_0,A_0)$. As $h(A_0)\leq h(A)$ and $\ell_\sigma(T_1)\leq \ell_\sigma(G)$, for proving parts~$(1)$ and~$(2)$ we may assume that $G=T_1$, $\sigma=\sigma_0$ and $A=A_0$.

Part~(1): We argue by induction on $h+|G|$. If $h=1$, then $\nu_\sigma(\mathfrak{T})=\beta_\sigma(\mathfrak{T})=1$ and the proof follows.

Assume that $\pi^*(P_h)\notin \sigma$. Write $\mathfrak{T}':=(P_i\mid i\in\{1,\ldots,h-1\})$. From Lemma~\ref{lemma33},  the family $\mathfrak{T}'$ is a tower of $G$.  As $\nu_{\sigma}(\mathfrak{T}')=\nu_{\sigma}(\mathfrak{T})$, $\beta_\sigma(\mathfrak{T}')=\beta_\sigma(\mathfrak{T})$ and $\mathfrak{T}'$ has length $h-1$, the proof follows by induction.

Assume that $\pi^*(P_h)\in \sigma$. Let $t\in\{1,\ldots, h\}$ with $T_t=P_hP_{h-1}\cdots P_t$ a $\sigma$-block of $\mathfrak{T}$. Suppose that $T_t$ is the only $\sigma$-block of $\mathfrak{T}$.  Thus
$\nu_\sigma(\mathfrak{T})=h-t+1$, $\beta_\sigma(\mathfrak{T})=1$ and $T_t$ is a Hall $\sigma$-subgroup of $G$. Moreover, since $T_t\unlhd T_1=G$, we have $A=T_t$.  Write $\mathfrak{T}':=(P_i\mid i\in \{t,\ldots,h\})$. From Lemma~\ref{lemma33}, the family $\mathfrak{T}'$ is a tower of $G$ and hence a tower of $A$. As $\mathfrak{T}'$ has length $h-t+1$, from Lemma~\ref{lemma31}, we get $h(A)\geq h-t+1$ and the proof follows.

Suppose that $T_t$ is not the only $\sigma$-block of $G$, and let $j\in \{1,\ldots,t-1\}$ be maximal with  $\pi^*(P_j)\in\sigma$. Suppose $\pi^*(P_j)\neq \pi^*(P_t)$. Then Lemma~\ref{lemma33} yields that $\mathfrak{T}':=(P_i\mid i\in \{1,\ldots,h\}\setminus\{j+1,\ldots,t-1\})$ is a tower of $G$. Since $\mathfrak{T}'$ has length $h-(t-j-1)<h$, from our induction we deduce
$$h(A)\geq \nu_\sigma(\mathfrak{T}')-\beta_\sigma(\mathfrak{T}')+1=\nu_\sigma(\mathfrak{T})-
(\beta_\sigma(\mathfrak{T})-1)+1=\nu_\sigma(\mathfrak{T})-\beta_\sigma(\mathfrak{T})+2.$$

Finally, suppose that $\pi^*(P_j)=\pi^*(P_t)$. In particular, either $\pi^*(P_{j-1})\neq 
\pi^*(P_t)$ or $j=1$. Now, Lemma~\ref{lemma33} gives that $\mathfrak{T}':=(P_i\mid i\in \{1,\ldots,h\}\setminus\{j,\ldots,t-1\})$ is a tower of $G$. As $\mathfrak{T}'$ has length $h-(t-j)<h$, the inductive hypothesis yields
$$h(A)\geq \nu_\sigma(\mathfrak{T}')-\beta_\sigma(\mathfrak{T}')+1=(\nu_\sigma(\mathfrak{T})-1)-
(\beta_\sigma(\mathfrak{T})-1)+1=\nu_\sigma(\mathfrak{T})-\beta_\sigma(\mathfrak{T})+1.$$

Part~(2): As in Part~$(1)$, we proceed by induction on $h+|G|$. Assume $\pi^*(P_h)\notin\sigma$. Then $\mathfrak{T}':=(P_i\mid i\in \{1,\ldots,h-1\})$ is a tower of $G$ of length $h-1$ with $\beta_\sigma(\mathfrak{T}')=\beta_\sigma(\mathfrak{T})$. Thus the proof follows  by induction.

Assume that $\pi^*(P_{h})\in \sigma$. Write $N:=\O {\sigma'} G$ and assume first that $N\neq 1$. For $j\in \{1,\ldots,h-1\}$, we have $[P_j\cap N,P_h]\leq N\cap P_h=1$ and hence $P_j\cap N\leq \cent {P_j}{P_h}$. In particular, by Lemma~\ref{elementaryy}, $\mathfrak{T}':=(P_iN/N\mid i\in \{1,\ldots,h-1\})$ is a tower of $G/N$ and, by induction, $\beta_\sigma(\mathfrak{T}')\leq \ell_{\sigma}(G/N)$. Since $\beta_\sigma(\mathfrak{T})=\beta_\sigma(\mathfrak{T}')$ and $\ell_\sigma(G)\geq \ell_{\sigma}(G/N)$, we get $\beta_\sigma(\mathfrak{T})\leq \ell_{\sigma}(G)$. Assume then that $N=1$. 

Write $\mathfrak{T}':=(P_i\mid i\in \{1,\ldots,h-1\})$. By Lemma~\ref{lemma33}, $\mathfrak{T}'$ is a tower of $G$ of length $h-1$. If $P_h$ is a not $\sigma$-block, then $\beta_\sigma(\mathfrak{T}')=\beta_\sigma(\mathfrak{T})$ and, by induction, $\beta_\sigma(\mathfrak{T})\leq \ell_\sigma(G)$.  
Suppose that $P_h$ is a $\sigma$-block, that is, $\pi^*(P_{h-1})\notin\sigma$. Clearly, $
\beta_\sigma(\mathfrak{T})=\beta_{\sigma}(\mathfrak{T}')+1$.

Write $M:=\O {\sigma}G$ and observe that $M\neq 1$ and 
$\ell_\sigma(G)=\ell_{\sigma}(G/M)+1$
because $\O {\sigma'} G=N=1$. For $j\in \{1,\ldots,h-2\}$, we have $[ P_j\cap M,P_{h-1}]\leq M\cap P_{h-1}=1$ and hence $P_j\cap M\leq \cent {P_j}{P_{h-1}}$.  In particular, by Lemma~\ref{elementaryy} (applied to $\mathfrak{T}'$), $\mathfrak{T}{''}:=(P_jM/M\mid j\in \{1,\ldots,h-2\})$ is a tower of $G/M$.  Now, by induction, $\ell_\sigma(G/M)\geq \beta_\sigma(\mathfrak{T}^{''})=\beta_\sigma(\mathfrak{T}')$ from which it follows that
$\ell_\sigma(G)\geq\beta_\sigma(\mathfrak{T})$.
\end{proof}

\section{Factorisations: Proofs of Theorems~\ref{thrmA} and~\ref{thrmB} and Corollary~\ref{corcor}}\label{section4}
We start by proving the following.

%\begin{lemma}\label{lemma41}
%Let $G$ be a group, let $\sigma$ be a non-empty proper subset of $\pi(G)$ and let $G=AB$ be a factorization, with $A$ a $\sigma$-subgroup of $G$ and $B$ a $\sigma'$-subgroup of $G$. Then
%\[
%h(G)\leq h(A)+h(B)+2\min\{\ell_\sigma(G),\ell_{\sigma'}(G)\}-1-\delta,
%\]
%where $\delta:=0$ if $\ell_\sigma(G)\neq\ell_{\sigma'}(G)$ and $\delta=1$ if $\ell_\sigma(G)=\ell_{\sigma'}(G)$.
%\end{lemma}

\begin{lemma}\label{lemma41}
Let $G$ be a group, let $\sigma$ be a non-empty proper subset of $\pi(G)$ and let $G=AB$ be a factorisation, with $A$ a $\sigma$-subgroup of $G$ and $B$ a $\sigma'$-subgroup of $G$. Then
\[
h(G)\leq h(A)+h(B)+\ell_\sigma(G)+\ell_{\sigma'}(G)-2
\]
and
\[
h(G)\leq h(A)+h(B)+2\min\{\ell_\sigma(G),\ell_{\sigma'}(G)\}-1.
\]
\end{lemma}
\begin{proof}
Let $\mathfrak{T}$ be a Fitting tower of $G$ (see Definition~\ref{Fittingtower}). Using first Lemma~\ref{lemma34} part~(1) and then part~(2), we have
\begin{eqnarray*}
(\dag)\qquad h(G)&=&\nu_\sigma(\mathfrak{T})+\nu_{\sigma'}(\mathfrak{T})\leq (h(A)+\beta_\sigma(\mathfrak{T})-1)+(h(B)+\beta_{\sigma'}(\mathfrak{T})-1)\\
&=&h(A)+h(B)+\beta_{\sigma}(\mathfrak{T})+\beta_{\sigma'}(\mathfrak{T})-2\\
&\leq &h(A)+h(B)+\ell_{\sigma}(G)+\ell_{\sigma'}(G)-2.
\end{eqnarray*}

Observe that, for each set of prime numbers $\pi$, from the definition of $\pi'\pi$-series we have $\ell_{\pi'}(G)\leq \ell_\pi(G)+1$. Applying this remark with $\pi=\sigma$ and  with $\pi=\sigma'$, from $(\dag)$ we get 
$h(G)\leq h(A)+h(B)+2\min\{\ell_\sigma(G),\ell_{\sigma'}(G)\}-1.$
\end{proof}

\begin{proof}[Proof of Theorem~$\ref{thrmA}$]
Write $\sigma:=\pi(A)$ and $\sigma':=\pi(B)$. If $|B|$ is odd or if $B$ is nilpotent, then Theorem~\ref{thrm2.3} yields $\ell_{\sigma'}(G)\leq d(B)$. In the first case, Eq.~\eqref{eq:1} follows directly from Lemma~\ref{lemma41}. In the second case, $h(B)=1$ and now Eq.~\eqref{eq:2} follows again from Lemma~\ref{lemma41}.  
\end{proof}

We now show that the bounds in Theorem~\ref{thrmA} are (in some cases) best possible. (We denote by $C_n$ a cyclic group of order $n$.)
\begin{example}\label{ex1}{\rm
Let $p,q,r$ and $t$ be distinct primes and let $n\geq 1$. Define $H_0:=C_p\wr C_q$ and $H_1:=(H_0\wr C_r)\wr (C_q\wr C_p)$. Now, for each $i\geq 1$, define inductively $H_{2i}:=(H_{2i-1}\wr C_r)\wr (C_p\wr C_q)$ and $H_{2i+1}:=(H_{2i}\wr C_r)\wr (C_q\wr C_p)$. 

We let $H:=H_{n}$ and $G:=C_t\wr H$. Let $A$ be a Hall $\{p,q\}$-subgroup of $G$ and let $B$ be a Hall $\{r,t\}$-subgroup of $G$.  A computation shows that $h(A)=n+2$, $h(B)=2$, $h(G)=3n+3$ and $d(B)=n+1$. Theorem~\ref{thrmA}~\eqref{eq:1} predicts $h(G)\leq h(A)+h(B)+2d(B)-1$, and in fact in this example the equality is met.
}
\end{example}

\begin{example}\label{ex2}{\rm 
Let $p$ and $q$ be distinct primes and let $n\geq 0$. Define $G_0:=C_p$ and $G_1:=G_0\wr C_q$. Now, for each $i\geq 1$, define inductively $G_{2i}:=G_{2i-1}\wr C_p$ and $G_{2i+1}:=G_{2i}\wr C_q$. 

Let $G:=G_{2n}$, let $A$ be a Sylow $p$-subgroup of $G$ and let $B$ be a Sylow $q$-subgroup of $G$. A computation shows that $h(A)=1$, $d(B)=n$ and $h(G)=2n+1$. Theorem~\ref{thrmA}~\eqref{eq:2} predicts $h(G)\leq h(A)+2d(B)$, and in fact in this example the equality is met.}
\end{example}

\begin{proof}[Proof of Corollary~$\ref{corcor}$] From Lemma~\ref{lemma:2.1}, there exists a Sylow basis $\{G_p\}_{p\in\pi(G)}$ of $G$ with $A=\prod_{p\in\pi(A)}G_p$ and $B=\prod_{p\in \pi(B)}G_p$. 

Now, we argue by induction on $|\pi(B)|$. If $|\pi(B)|=1$, then $B$ is nilpotent and hence the proof follows from Theorem~\ref{thrmA}~\eqref{eq:2}. Suppose that $|\pi(B)|>1$. Fix $q\in \pi(B)$ and write $B_{q'}:=\prod_{p\in \pi(B)\setminus\{q\}}G_p$. Clearly, $G=AB=(AG_q)B_{q'}$ and hence, by induction, 
\begin{eqnarray*}
h(G)&\leq& h(AG_q)+2\sum_{p\in \pi(B_{q'})}d(G_p)\leq (h(A)+2d(G_q))+2\sum_{p\in\pi(B_{q'})}d(G_p)\\
&=&h(A)+2\sum_{p\in \pi(B)}d(G_p).
\end{eqnarray*}
\end{proof}

The proof of Theorem~\ref{thrmB} will follow at once from the following lemma, which (in our opinion) is of independent interest.

\begin{lemma}\label{lemma4.4}
Let $G$ be a group, let $\sigma$ be a non-empty subset of $\pi(G)$ and let $H$ be a Hall $\sigma$-subgroup of $G$. Then $\ell_\sigma(G)\leq \delta(H)h(H)$.
\end{lemma}
\begin{proof}
When $|\sigma|=1$, the proof follows immediately from Theorem~\ref{thrm2.2}. In particular, we may assume that $|\sigma|>1$. Now we proceed by induction on $|G|+|\sigma|$.

Clearly, $\ell_\sigma(G)=\ell_\sigma(G/\O{\sigma'}G)$ and $H\O{\sigma'} G/\O{\sigma'}G\cong H$ is a Hall $\sigma$-subgroup of $G/\O{\sigma'}G$. When $\O{\sigma'}G\neq 1$, the proof follows by induction, and hence we may assume that $\O{\sigma'}G=1$.

Suppose that $G$ contains two distinct minimal normal subgroups $N$ and $M$. Clearly, $\pi(N),\pi(M)\subseteq \sigma$. As $\O{\sigma'}G=1$, we deduce that $\ell_\sigma(G/N)=\ell_\sigma(G)=\ell_\sigma(G/M)$. Moreover, by induction, $\ell_\sigma(G/N)\leq \delta(H/N)h(H/N)\leq \delta(H)h(H)$. This gives $\ell_\sigma(G)\leq\delta(H)h(H)$, and hence we may assume that $G$ contains a unique minimal normal subgroup. This yields $\F G=\O p G$, for some $p\in \sigma$. As $\cent G{\O p G}\leq \O p G$ and $\O p G\leq H$, we have $\F H=\O  p H$.

Write $\tau:=\sigma\setminus\{p\}$. Observe that $\ell_\sigma(G)\leq \ell_p(G)+\ell_\tau(G)$. As $G_\tau$ is isomorphic to a subgroup of $H/\F H$, we get $h(G_\tau)\leq h(H/\F H)=h(H)-1$. Now, from the inductive hypothesis, we obtain
\begin{eqnarray*}
\ell_\sigma(G)&\leq& \ell_p(G)+\ell_\tau(G)\leq \delta(G_p)h(G_p)+\delta(G_\tau)h(G_\tau)\\
&\leq& \delta(H)+\delta(H)(h(H)-1)\leq \delta(H)h(H).
\end{eqnarray*}
\end{proof}

\begin{proof}[Proof of Theorem~$\ref{thrmB}$]
Write $\sigma:=\pi(A)$ and $\sigma':=\pi(B)$. From Lemma~\ref{lemma4.4}, we get $\ell_\sigma(G)\leq \delta(A)h(A)$ and $\ell_{\sigma'}(G)\leq \delta(B)h(B)$. Now the proof follows from the second inequality in Lemma~\ref{lemma41}.
\end{proof}

\section{Factorisations: Proof of Theorem~\ref{thrmC}}\label{section5}
Before proceeding with the proof of Theorem~\ref{thrmC} we need to introduce some auxiliary notation.

Given a group $G$, we denote with $\R G $ the \textit{nilpotent residual} of $G$, that is, the smallest (with respect to inclusion) normal subgroup $N$ of $G$ with $G/N$ nilpotent. Then, we define inductively the descending normal series $\{\RR i G\}_i$ by $\RR 0 G:=G$ and $\RR {i+1} G:=\R{\RR i G}$, for every $i\geq 0$. Observe that if $h=h(G)$, then for every $i\in \{0,\ldots,h\}$ we have  $\RR {h-i}G\leq \FF i G$.

Now, let $A$ be a Hall subgroup of $G$  and, for $i\in \{1,\ldots,h\}$, define
\[
\ell^i(G,A):=\max\{\ell_p(G)\mid p\in \pi(\RR{i-1} A /\RR i A)\}\quad\textrm{and}\quad \Lambda_G(A):=\sum_{i=1}^{h(A)}\ell^i(G,A).
\]
It is clear  that, for every normal subgroup $N$ of $G$, $\Lambda_{G/N}(AN/N)\leq \Lambda_G(A)$.

\begin{lemma}\label{thrm48}Let $G=AB$ be a finite soluble group factorised by its proper subgroups $A$ and $B$ with $\gcd(|A|,|B|)=1$. Then $h(G)\leq \Lambda_G(A)+\Lambda_G(B)$.
\end{lemma}

\begin{proof}
We argue by induction on $|G|$. Suppose that $G$ contains two distinct minimal normal subgroups $N$ and $M$. Clearly, $h(G/N)=h(G)=h(G/M)$ and hence by induction $h(G)\leq \Lambda_{G/N}(AN/N)+\Lambda_{G/N}(BN/N)\leq \Lambda_G(A)+\Lambda_G(B)$. In particular, we may assume that $G$ contains a unique minimal normal subgroup $N$ and, replacing $A$ by $B$ if necessary, that $\{p\}=\pi(N)\subseteq \pi(A)$.  This yields $\F G=\O p G$. As $\cent G{\O p G}\leq \O p G$ and $\O p G\leq A$, we have $\F A=\O  p A$.

Write $h:=h(A)$. Now, $\RR{h-1}A\leq \FF 1 A=\F A$ and hence $\RR {h-1}A$ is a $p$-group. Thus
$$\Lambda_G(A)=\ell_p(G)+\sum_{i=1}^{h-1}\ell^i(G,A).$$
Since $\ell_p(G/\F G)=\ell_p(G)-1$, we get $\Lambda_{G/\F G}(A/\F G)\leq \Lambda_G(A)-1$. Moreover, since $p\notin \pi(B)$, we have $\Lambda_{G/\F G}(B\F G/\F G)=\Lambda_{G}(B)$. Therefore the inductive hypothesis gives
\begin{eqnarray*}
h(G)&=&h(G/\F G)+1\\
&\leq& \Lambda_{G/\F G}(A\F G/\F G)+\Lambda_{G/\F G}(B\F G/\F G)+1\\
&\leq& \Lambda_G(A)+\Lambda_G(B),
\end{eqnarray*}
and the proof is complete.
\end{proof}

\begin{proof}[Proof of Theorem~$\ref{thrmC}$]
For each $p\in \pi(A)$, Theorem~\ref{thrm2.2} yields $\ell_p(G)\leq d(G_p)$ and hence $\ell_p(G)\leq \delta(A)$. It follows that $\Lambda_G(A)\leq \delta(A)h(A)$. The same argument applied to $B$ gives $\Lambda_G(B)\leq \delta(B)h(B)$. Now the proof follows from Lemma~\ref{thrm48}.
\end{proof}

A weaker form of Theorem~\ref{thrmC} can be deduced from the results in Section~\ref{section4}. Indeed, from the first inequality in Lemma~\ref{lemma41} and from Lemma~\ref{lemma4.4}, we get
\begin{eqnarray*}
h(G)&\leq &h(A)+h(B)+\ell_{\sigma}(G)+\ell_{\sigma'}(G)-2\\
&\leq& h(A)+h(B)+\delta(A)h(A)+\delta(B)h(B)-2\\
&=&(\delta(A)+1)h(A)+(\delta(B)+1)h(B)-2.
\end{eqnarray*}
Clearly  Theorem~\ref{thrmC} always offer a better estimate on $h(G)$.

\end{document}